\documentclass[12pt,draft]{amsart}

\address{Department of Mathematics\\Boyd Graduate Studies Research Center\\University of Georgia\\Athens, GA 30602\\USA}
\author{Emily Jennings}
\email{emily@math.uga.edu}

\author{Paul Pollack}
\email{pollack@uga.edu}

\author{Lola Thompson}
\email{lola@math.uga.edu}

\title[Variations on a theorem of Davenport]{Variations on a theorem of Davenport concerning abundant numbers}
\usepackage{amsmath,amssymb,amsthm,url}
\usepackage{geometry}
\DeclareMathAlphabet{\curly}{U}{rsfs}{m}{n}

\newtheorem{thm}{Theorem}[section]
\newtheorem{cor}[thm]{Corollary}

\newtheorem{prop}[thm]{Proposition}
\newtheorem{lem}[thm]{Lemma}
\theoremstyle{definition}

\theoremstyle{remark}
\newtheorem*{ex}{Examples}
\numberwithin{equation}{section}
\begin{document}
\renewcommand{\labelenumi}{(\roman{enumi})}
\def\d{\mathrm{d}}
\def\i{\mathrm{i}}
\def\e{\mathrm{e}}
\def\C{\mathbf{C}}
\def\D{\curly{D}}
\def\A{\curly{A}}
\def\Pp{\curly{P}}
\def\Prob{\mathbf{P}}
\def\E{\mathbf{E}}
\def\R{\mathbf{R}}
\def\M{\mathfrak{M}}
\def\N{\mathbf{N}}
\def\Q{\mathbf{Q}}
\def\Z{\mathbf{Z}}
\def\Ss{\mathcal{S}}
\def\1{\mathbf{1}}
\def\tD{\tilde{D}}
\def\lcm{\mathop{\operatorname{lcm}}}
\newcommand{\leg}[2]{\genfrac{(}{)}{}{}{#1}{#2}}
\makeatletter
\def\moverlay{\mathpalette\mov@rlay}
\def\mov@rlay#1#2{\leavevmode\vtop{%
   \baselineskip\z@skip \lineskiplimit-\maxdimen
   \ialign{\hfil$\m@th#1##$\hfil\cr#2\crcr}}}
\newcommand{\charfusion}[3][\mathord]{
    #1{\ifx#1\mathop\vphantom{#2}\fi
        \mathpalette\mov@rlay{#2\cr#3}
      }
    \ifx#1\mathop\expandafter\displaylimits\fi}
\makeatother

\makeatletter
\let\@@pmod\pmod
\DeclareRobustCommand{\pmod}{\@ifstar\@pmods\@@pmod}
\def\@pmods#1{\mkern4mu({\operator@font mod}\mkern 6mu#1)}
\makeatother

\newcommand{\cupdot}{\charfusion[\mathbin]{\cup}{\cdot}}
\newcommand{\bigcupdot}{\charfusion[\mathop]{\bigcup}{\cdot}}
\begin{abstract} Let $\sigma(n) = \sum_{d \mid n}d$ be the usual sum-of-divisors function. In 1933, Davenport  showed that that $n/\sigma(n)$ possesses a continuous distribution function. In other words, the limit $D(u):= \lim_{x\to\infty} \frac{1}{x}\sum_{n \leq x,~n/\sigma(n) \leq u} 1$  exists for all $u \in [0,1]$ and varies continuously with $u$.
We study the behavior of the sums $\sum_{n \leq x,~n/\sigma(n) \leq u} f(n)$ for certain complex-valued multiplicative functions $f$. Our results cover many of the more frequently encountered functions, including $\varphi(n)$, $\tau(n)$, and $\mu(n)$. They also apply to the representation function for sums of two squares, yielding the following analogue of Davenport's result: For all $u \in [0,1]$, the limit
\[ \tilde{D}(u):= \lim_{R\to\infty} \frac{1}{\pi R}\#\{(x,y) \in \Z^2: 0<x^2+y^2 \leq R \text{ and } \frac{x^2+y^2}{\sigma(x^2+y^2)} \leq u\} \]
exists, and $\tilde{D}(u)$ is both continuous and strictly increasing on $[0,1]$.
\end{abstract}
\maketitle

\section{Introduction}
Recall that a natural number $n$ is said to be \emph{abundant} if $\sigma(n) > 2n$, where $\sigma(n):=\sum_{d\mid n}d$ denotes the usual sum-of-divisors function. Answering a question of Bessel-Hagen, Davenport \cite{davenport33} showed that the set of abundant numbers possesses an asymptotic density. In fact, he proved the more precise result that $n/\sigma(n)$ possesses a continuous distribution function. In other words, the limit
\begin{equation}\label{eq:davenportexpression} D(u):= \lim_{x\to\infty} \frac{1}{x}\sum_{\substack{n \leq x \\ n/\sigma(n) \leq u}} 1 \end{equation}
exists for all $u \in [0,1]$ and varies continuously with $u$. We have followed modern conventions in writing the condition on $n/\sigma(n)$ as a non-strict inequality, but since $D(u)$ is continuous, whether or not we allow $n/\sigma(n)=u$ does not change the value of $D(u)$. Recent work of Kobayashi \cite{kobayashi10} (see also \cite{kobayashi13}) shows that $0.24761 < D(\frac12) < 0.24765$, so that just under $1$ in $4$ numbers are abundant.

The purpose of this paper is to establish analogues of Davenport's theorem where the uninteresting summand $1$ appearing in \eqref{eq:davenportexpression} is replaced with $f(n)$ for certain complex-valued multiplicative functions $f$. We prove two theorems in this direction, the first of which is as follows. Recall that an arithmetic function $f$ is said to possess a \emph{mean value} if $\frac{1}{x}\sum_{n\le x}f(n)$ approaches a (complex number) limit as $x\to\infty$.

\begin{thm}\label{thm:smallmult} Let $f$ be a multiplicative function that is bounded in mean square, i.e.,
	\[ \limsup_{x\to\infty} \frac{1}{x}\sum_{n \leq x}|f(n)|^2 < \infty. \]
	Suppose that for every nonnegative integer $k$, the function $n\mapsto f(n) \left(n/\sigma(n)\right)^k$ possesses a mean value. Then for every real $u \in [0,1]$, the limit
	\begin{equation}\label{eq:dfdef} D_f(u):= \lim_{x\to\infty} \frac{1}{x}\sum_{\substack{n \leq x \\ n/\sigma(n)\leq u}} f(n) \end{equation}
	exists. Moreover, $D_f(u)$ is continuous as a function of $u$.
\end{thm}

Theorem \ref{thm:smallmult} is proved in \S\ref{sec:smallmult}. In the same section, we obtain the following consequences. From now on, let $p$ be a prime variable.

\begin{cor}\label{cor:usual} Let $f$ be a multiplicative function bounded in mean square. Then the hypotheses of Theorem \ref{thm:smallmult}, and hence also its conclusion, hold if
\begin{equation}\label{eq:smallprod} \sum_{p} \frac{|f(p)-1|}{p} < \infty \quad\text{and}\quad \sum_{p} \sum_{j\geq 2} \frac{|f(p^j)|}{p^j} < \infty. \end{equation}
If $|f(n)| \leq 1$ for all $n \in \N$, then  \eqref{eq:smallprod} can be replaced with the weaker assumption that the series
\begin{equation}\label{eq:delange} \sum_{p} \frac{f(p)-1}{p} \end{equation}
converges (possibly conditionally).
\end{cor}

\begin{cor}\label{cor:meanzero} Let $f$ be a multiplicative function with $|f(n)| \leq 1$ for all natural numbers $n$. Suppose that $f$ has mean value zero. Suppose further that there is no real number $\beta$ with the property that $f(2^j)= -2^{\i j \beta}$ for every positive integer $j$. Then the function $D_f(u)$ defined in \eqref{eq:dfdef} vanishes identically for all $u \in [0,1]$.
\end{cor}

\begin{ex}\mbox{ }
\begin{enumerate}
\item A simple example of a function satisfying the hypotheses of Corollary \ref{cor:usual} is the indicator function of the squarefree numbers (or more generally, the $\ell$-free numbers). The hypotheses of that result also hold for the functions $(\varphi(n)/n)^z$ and $(\sigma(n)/n)^z$, for any complex number $z$.  To obtain a result for $\varphi(n)$ or $\sigma(n)$, one can apply Corollary \ref{cor:usual} to $\varphi(n)/n$ or $\sigma(n)/n$, and then remove the weight of $1/n$ by partial summation. Indeed, whenever the conclusion of Theorem \ref{thm:smallmult} holds,  \[ \lim_{x\to\infty}\frac{1}{x^2}\sum_{\substack{n \leq x \\ n/\sigma(n)\leq u}} n f(n) = \frac{1}{2} D_{f}(u). \]

\item A natural family of examples satisfying the hypotheses of Corollary \ref{cor:meanzero} are the functions $\lambda_{a,q}(n):= \exp(2\pi \i\frac{a}{q}\Omega(n))$ with $q$ not dividing $a$. Here, as usual, $\Omega(n)$ denotes the number of prime factors of $n$ counted with multiplicity. That all of the functions $\lambda_{a,q}(n)$ have mean value zero seems to have been first proved by Pillai and Chowla \cite{pillai40} (alternatively, this assertion follows from a beautiful theorem of Hal\'asz, quoted in \S\ref{sec:smallmult}). The conclusion of Corollary \ref{cor:meanzero} for this family leads, via the orthogonality relations for additive characters, to the following pretty consequence:
\begin{quotation}
\emph{Fix $q\in \N$ and fix $0 < u \leq 1$. As $n$ ranges over the solutions to $n/\sigma(n)\leq u$, the values $\Omega(n)$ are equidistributed mod $q$.}
\end{quotation}
The nontrivial Dirichlet characters form another natural class of examples. Here the corresponding conclusion is:
\begin{quotation}
\emph{Fix $q\in \N$ and fix $0 < u \leq 1$. The solutions $n$ to $n/\sigma(n) \leq u$ that are relatively prime to $q$ are equidistributed among the coprime residue classes modulo $q$.}
\end{quotation}
Actually, for this deduction to be valid, one must know that a positive proportion of solutions
 to $n/\sigma(n)\leq u$ are coprime to $q$. This will follow from Theorem \ref{thm:nonneg} below. A different
  proof of this equidistribution result was indicated in \cite{pollack}.
\end{enumerate}
\end{ex}

For our second theorem, we restrict attention to nonnegative functions $f$ (assumed not to vanish identically). While Theorem \ref{thm:smallmult} applies perfectly well to many nonnegative $f$, for others it is simply not the right tool for the job. An illustrative example is provided by the divisor function $\tau$. The mean value of $\tau$ on the interval $[1,x]$ is asymptotic to $\log{x}$, as $x\to\infty$. Thus, to obtain the `correct' analogue of Davenport's theorem,  we should not be dividing by $x$ in \eqref{eq:dfdef} but rather by something proportional to $x\log{x}$. More generally, for a nonnegative function $f$, we ought to normalize by the factor
\[ S(f;x):= \sum_{n \leq x} f(n). \]
We are thus led to define
\[ \tD_{f}(u)= \lim_{x\to\infty} \frac{1}{S(f;x)} \sum_{\substack{n\leq x\\ n/\sigma(n)\leq u}} f(n), \]
whenever the limit exists. We can now state our second main result.

\begin{thm}\label{thm:nonneg} Suppose that $f$ is a nonnegative multiplicative function with the property that as $x\to\infty$,
\begin{equation}\label{eq:wcond}\sum_{p \leq x} f(p) \frac{\log{p}}{p} \sim \kappa \log{x} \end{equation}
for some $\kappa > 0$. Suppose also that $f(p)$ is bounded for primes $p$ and that \begin{equation}\label{eq:twoormore} \sum_{p}\sum_{j\geq 2} \frac{f(p^j)}{p^j} < \infty. \end{equation} If $\kappa \leq 1$, suppose further that \[ \sum_{p^j \leq x} f(p^j) \ll_{f} x/\log{x} \quad \text{\rm{(}for $x\geq 2$\rm{)}}.\]
Then $\tD_f(u)$ exists for all $u \in [0,1]$ and is both continuous and strictly increasing.
\end{thm}

\begin{ex}\mbox{ }
\begin{enumerate}
\item When $f=\tau$, the hypotheses of Theorem \ref{thm:nonneg} hold with $\kappa=2$.
\item Let $r(n)=\frac{1}{4}\#\{(x,y)\in \Z^2: x^2+y^2=n\}$. This function fails the hypotheses of Theorem \ref{thm:smallmult} (by not being bounded in mean square), but it satisfies the hypotheses of Theorem \ref{thm:nonneg} with $\kappa=1$. Since $\sum_{n \leq x} r(n) \sim \frac{\pi}{4}x$ by simple geometric considerations (see \cite[Theorem 339, p. 357]{HW08}), we see that
\[ \tD_{r}(u)=\lim_{R\to\infty} \frac{1}{\pi R}\#\{(x,y) \in \Z^2: 0<x^2+y^2 \leq R \text{ and } \frac{x^2+y^2}{\sigma(x^2+y^2)} \leq u\}. \]
The existence and continuity of $\tD_{r}(u)$ may be thought of as a sum-of-two-squares analogue of Davenport's result.
\item \emph{Multiplicative sets} provide a rich source of examples. Here a set $\mathcal{S}$ of natural numbers is called multiplicative if its indicator function $\1_\mathcal{S}$ is multiplicative. Suppose that $\Ss$ is multiplicative and contains a well-defined, positive proportion of the primes, in the sense that \eqref{eq:wcond} holds with $f=\1_\Ss$ and a certain $\kappa >0$. (This notion of the density of a set of primes is weaker than that of natural density.) Then Theorem \ref{thm:nonneg} shows that $n/\sigma(n)$ has a continuous, strictly increasing distribution function when restricted to $\mathcal{S}$.

As a concrete example, we may take $\Ss$ to be the set of sums of two squares (where $\kappa=\frac12$). We thus obtain another two-squares analogue of Davenport's result, this time with the elements of $\Ss$ counted without multiplicity.
\end{enumerate}
\end{ex}

\subsection*{Notation} We use an upright letter $\mathrm{e}$ for the constant $2.71828\dots$, and we (continue to) use $\i$ for the imaginary unit. If $F$ is a function on $[0,1]$, we write $\|F\|_{\infty}$ for the $L^{\infty}$-norm of $F$. We employ $O$ and $o$-notation, as well as the associated Vinogradov symbols $\ll$ and $\gg$, with the usual meanings. All implied constants are absolute unless the dependence is explicitly indicated (e.g., with a subscript).

\section{Proof of Theorem \ref{thm:smallmult}}\label{sec:smallmult}
We first show the existence of the limit \eqref{eq:dfdef} when the sharp cut-off condition $n/\sigma(n)\leq u$ is `smoothed out'.

\begin{lem}\label{lem:continuous} Let $f$ be a multiplicative function satisfying the hypotheses of Theorem \ref{thm:smallmult}. For every continuous function $\psi$ on $[0,1]$, the limit
\[ \lim_{x\to\infty} \frac{1}{x}\sum_{n \leq x} f(n) \psi\left(\frac{n}{\sigma(n)}\right) \]
exists.
\end{lem}

\begin{proof} Since $\psi$ is continuous on $[0,1]$, the Weierstrass approximation theorem allows us to choose a sequence of polynomials $p_m(x)$ with $\|\psi - p_m\|_{\infty} \leq \frac{1}{m}$.  Since the arithmetic function $f(n) (n/\sigma(n))^k$ has a mean value for all nonnegative integers $k$, it follows that
	\[ \mu_m:= \lim_{x\to\infty}\frac{1}{x} \sum_{n \leq x} f(n) p_m\left(\frac{n}{\sigma(n)}\right) \]
	exists for each $m$. In fact, the sequence $\{\mu_m\}$ is Cauchy. To see this, we start by observing that
\begin{equation}\label{eq:cauchy0} |\mu_m - \mu_{m'}| \leq \|p_m-p_{m'}\|_{\infty}\cdot \limsup_{x\to\infty}\frac{1}{x} \sum_{n \leq x} |f(n)| \leq \frac{2}{\min\{m,m'\}}\cdot \limsup_{x\to\infty}\frac{1}{x} \sum_{n \leq x} |f(n)|. \end{equation}
Since $f$ is bounded in mean square,  Cauchy--Schwarz shows that
\begin{equation}\label{eq:cauchy1} \limsup_{x\to\infty}\frac{1}{x} \sum_{n \leq x} |f(n)| \leq \limsup_{x\to\infty} \left(\frac{1}{x}\sum_{n \leq x}|f(n)|^2\right)^{1/2} \ll_{f} 1.\end{equation}
Hence, $|\mu_m-\mu_{m'}| \ll_f \min\{m,m'\}^{-1}$, and so $\{\mu_m\}$ is a Cauchy sequence.
Let $\mu = \lim_{m\to\infty} \mu_m$. We claim that the limit in the statement of the lemma is precisely $\mu$. In fact, for every natural number $m$,
\begin{align*} \limsup_{x\to\infty} \bigg|\frac{1}{x}\sum_{n \leq x} f(n)& \psi\left(\frac{n}{\sigma(n)}\right) - \mu\bigg| \\&\leq |\mu-\mu_m| + \limsup_{x\to\infty} \left|\frac{1}{x}\sum_{n \leq x} f(n) \left(\psi\left(\frac{n}{\sigma(n)}\right) - p_m\left(\frac{n}{\sigma(n)}\right)\right)\right|
	\\ &\leq |\mu-\mu_m| + \|\psi-p_m\|_{\infty}\cdot \limsup_{x\to\infty}\frac{1}{x} \sum_{n \leq x} |f(n)| \ll_{f} \frac{1}{m},\end{align*}
using \eqref{eq:cauchy0} and \eqref{eq:cauchy1} in the last step. Since $m$ can be taken arbitrarily large, it follows that $\frac{1}{x}\sum_{n \leq x} f(n) \psi(n/\sigma(n)) \to \mu$, as desired.
\end{proof}

\begin{proof}[Proof of Theorem \ref{thm:smallmult}] We start by showing the existence of $D_f(u)$ for all $u \in [0,1]$, leaving the proof that $D_f(u)$ is continuous to the end. Since $D_f(1)$ is simply the mean value of $f$, we may assume that $0 \leq u < 1$.
Let $\psi$ be the characteristic function of $[0,u]$. Since $\psi$ is not continuous, we cannot directly apply Lemma \ref{lem:continuous}. To work around this, we define, for positive integers $m$ large enough that $u+\frac1m < 1$, functions
\[ \psi_m(x):= \begin{cases} 1&\text{if $0 \leq x \leq u$}, \\
1-m(x-u)&\text{if $u < x <u+\frac1m$,}\\0&\text{if $u+\frac1m \leq x \leq 1$}.
\end{cases}
\]
Since each $\psi_m$ is continuous, Lemma \ref{lem:continuous} assures the existence of
\[ \mu_m = \lim_{x\to\infty} \sum_{n \leq x}f(n) \psi_m\left(\frac{n}{\sigma(n)}\right). \]
For $m' > m$, we see that $\psi_{m'}-\psi_m$ is supported on $[u,u+\frac{1}{m}]$ and that $\|\psi_m-\psi_{m'}\|_{\infty}\leq 1$. Hence,
\begin{align}\notag |\mu_m - \mu_{m'}| &\leq \limsup_{x\to\infty}\frac{1}{x}\sum_{\substack{n \leq x \\u \leq n/\sigma(n) \leq u + \frac{1}{m}}}|f(n)|\\
	&\ll_{f} \limsup_{x\to\infty} \bigg(\frac{1}{x}\sum_{\substack{n \leq x \\u \leq n/\sigma(n) \leq u + \frac{1}{m}}}1\bigg)^{1/2} = \left(D\left(u+\frac1m\right)-D(u)\right)^{1/2}.\label{eq:largepairs}
\end{align}	
Since $D$ is continuous, the final expression tends to $0$ as $m$ tends to infinity.
Thus, the sequence of $\mu_m$ is Cauchy with limit $\mu$, say. Notice that
\begin{align*} \limsup_{x\to\infty} \bigg|\frac{1}{x}\sum_{n \leq x} f(n) &\psi\left(\frac{n}{\sigma(n)}\right) - \mu\bigg| \\&\leq |\mu-\mu_m| +  \limsup_{x\to\infty} \bigg|\frac{1}{x}\sum_{n \leq x} f(n) \left(\psi\left(\frac{n}{\sigma(n)}\right)-\psi_m\left(\frac{n}{\sigma(n)}\right)\right)\bigg|.\end{align*}
Now $\psi-\psi_m$ is supported on $[u,u+1/m]$, and $\| \psi-\psi_m\|_{\infty} \leq 1$; mimicking the process that led to \eqref{eq:largepairs}, we see that the right-hand $\limsup$ is $O_{f} ((D(u+1/m)-D(u))^{1/2})$. From \eqref{eq:largepairs}, we also have $\mu-\mu_m \ll_{f} (D(u+1/m)-D(u))^{1/2}$. Since $m$ can be taken arbitrarily large, we conclude that the limit defining $D_f(u)$ exists and equals $\mu$.

The continuity of $D_f$ is now easy and in fact was essentially handled above. Indeed, a computation analogous to \eqref{eq:largepairs} shows that for every $u, v \in [0,1]$, we have $D_f(u)-D_f(v) \ll_{f} |D(u)-D(v)|^{1/2}$. Since $D$ is continuous on $[0,1]$, it follows that $D_f$ is continuous as well.\end{proof}

Corollary \ref{cor:usual} will be deduced from the following two results. Proposition \ref{prop:wintner}, which admits a completely elementary proof, is due essentially to Wintner (see, for example, \cite[Corollary 2.3, pp. 51--52]{SS94}). Proposition \ref{prop:delange}, which lies substantially deeper, was first established by Delange \cite{delange61} (compare with \cite[Theorem 1.1, p. 234]{SS94}).

\begin{prop}\label{prop:wintner} Let $f$ be a multiplicative function satisfying \eqref{eq:smallprod}. Then $f$ has a mean value. This mean value can be expressed explicitly as
\begin{equation}\label{eq:explicitmean} \prod_{p} \left(1-\frac{1}{p}\right)\left(1+\frac{f(p)}{p} +\frac{f(p^2)}{p^2}+\dots\right). \end{equation}
\end{prop}

\begin{prop}\label{prop:delange} Let $f$ be a multiplicative function with $|f(n)|\leq 1$ for all $n \in \N$. If the series
	\[ \sum_{p} \frac{1-f(p)}{p} \]
	converges, then $f$ has a mean value, again given by \eqref{eq:explicitmean}.
\end{prop}

\begin{proof}[Proof of Corollary \ref{cor:usual}] Suppose first that $f$ is bounded in mean square and that \eqref{eq:smallprod} is satisfied. For each nonnegative integer $k$, let $f_k(n) = f(n) (n/\sigma(n))^k$. (Thus, $f=f_0$.) Since $|f_k(p^j)| \leq |f(p^j)|$, the double series in \eqref{eq:smallprod} remains convergent if $f$ is replaced by any of the $f_k$. Since $f_k(p) = f(p) + O_k(|f(p)|/p)$ and $\sum_{p}|f(p)-1|/p$ converges, to show that $\sum_{p}|f_k(p)-1|/p$ converges, it is enough to show that $\sum_{p}|f(p)|/p^2$ converges. But this is clear, since
\[ \sum_{p} \frac{|f(p)|}{p^2} \leq \sum_{p}\frac{|f(p)-1|}{p^2} + \sum_{p}\frac{1}{p^2} < \infty. \]
So by Proposition \ref{prop:wintner}, each $f_k$ possesses a mean value. This shows that the hypotheses of Theorem \ref{thm:smallmult} hold for $f$.

Now let us assume instead that $|f(n)| \leq 1$ for all $n$ and that the series \eqref{eq:delange} converges. With $f_k$ defined as in the last paragraph, each $f_k$ is a multiplicative function taking values in the unit disc. Since $f_k(p) = f(p) + O_k(1/p)$ and \eqref{eq:delange} converges, the series $\sum_{p}\frac{1-f_k(p)}{p}$ also converges. So by Proposition \ref{prop:delange}, each $f_k$ has a mean value. Since $f$ is clearly bounded in mean square, the hypotheses of Theorem \ref{thm:smallmult} are satisfied.
\end{proof}

To prove Corollary \ref{cor:meanzero}, we make use of a celebrated theorem of Hal\'asz \cite{halasz68} (for other expositions, see \cite[Chapter 6]{elliott79} or \cite[Theorem 3.1, p. 304]{SS94}).

\begin{prop}\label{prop:halasz} Suppose that $f$ is a multiplicative function satisfying $|f(n)| \le 1$ for all $n \in \N$.  Then $f$ has mean value zero if and only if one of the following holds:
\begin{enumerate}
\item[(i)] There is a real number $\beta$ so that $f(2^j) = -2^{\i j\beta}$ for each positive integer $j$. Moreover, the series
    \begin{equation}\label{eq:hseries} \sum_{p} \frac{1-\Re(f(p)p^{-\i \beta})}{p} \end{equation}
converges for this $\beta$.
\item[(ii)] The series \eqref{eq:hseries} diverges for every real $\beta$.
\end{enumerate}
\end{prop}

\begin{proof}[Proof of Corollary \ref{cor:meanzero}] This will be a corollary of the proof of Theorem \ref{thm:smallmult}, rather than the result itself. As above, let $f_k(n):= f(n) (n/\sigma(n))^k$. Since $f$ has mean value zero, but there is no $\beta$ with $f(2^j) = -2^{\i j\beta}$ for all $j$, it must be that \eqref{eq:hseries} diverges for every real $\beta$. Since $f_k(p) = f(p) + O(1/p)$, the series \eqref{eq:hseries} remains divergent for every real $\beta$ if $f$ is replaced by any of the $f_k$. So by Proposition \ref{prop:halasz} again, each $f_k$ has mean value zero.

Referring back to the proof of Lemma \ref{lem:continuous}, it follows that if $\psi$ is any continuous function on $[0,1]$, then $\frac{1}{x}\sum_{n \leq x} f(n)\psi(n/\sigma(n))\to 0$. Now referring to the proof of Theorem \ref{thm:smallmult}, we see that $D_f(u)$ vanishes identically, as desired.
\end{proof}

\section{Proof of Theorem \ref{thm:nonneg}}\label{sec:nonneg} Let $f$ be a nonnegative multiplicative function satisfying the conditions of Theorem \ref{thm:nonneg}. For each real $x \geq 1$, we introduce the distribution function
\begin{equation}\label{eq:fxchoice} F_x(u)= \frac{1}{S(f;x)} \sum_{\substack{n \leq x \\ \log(n/\sigma(n)) \le u}} f(n). \end{equation}
The reason for working with $\log(n/\sigma(n))$ instead of directly with $n/\sigma(n)$ is to ensure that the characteristic function of $F_x$ is amenable to analysis; this will be important later. Theorem \ref{thm:nonneg} is equivalent to the claim that the $F_x$ converge weakly to a continuous distribution function $F$ that is strictly increasing on $(-\infty,0]$. Indeed, $\tD_f$ and $F$ are related by the change of variables $\tD_f(\e^u) = F(u)$.

Our attack proceeds in three stages. First, we show the existence of the limiting distribution $F$. Next, we prove the continuity of $F$. Finally, we establish that $F$ is strictly increasing.

\subsection{Existence} We will apply L\'evy's convergence theorem, a well-known result drawn from the probabilist's toolchest (see, for example, \cite[Corollary 1, p. 350]{billingsley95}).

\begin{prop} Suppose that $\{F_x\}$ is any collection of distribution functions indexed by real numbers $x\geq 1$. For each $x\geq 1$, let $\phi_x(t)$ be the characteristic function of $F_x$. The following two statements are equivalent.
\begin{enumerate}
\item[(i)] The $F_x$ converge weakly to a distribution function $F$, as $x\to\infty$.
\item[(ii)] As $x\to\infty$, the $\phi_x$ converge pointwise on all of $\R$ to a function $\psi$ that is continuous at $0$.
\end{enumerate}
When \emph{(ii)} holds, $\psi$ is the characteristic function of the limiting distribution $F$.
\end{prop}

To evaluate the limit of the $\phi_x$ for our choice \eqref{eq:fxchoice} of $\{F_x\}$, we need the following versatile theorem of Wirsing \cite[Satz 1.1.1]{wirsing67}.

\begin{prop}\label{prop:wirsing} Suppose that $f$ is a complex-valued multiplicative function with the property that as $x\to\infty$,
\[ \sum_{p \leq x} f(p) \frac{\log{p}}{p} \sim \kappa \log{x} \]
for some real $\kappa > 0$. Suppose also that $f(p)$ is bounded and that \[ \sum_{p}\sum_{j \geq 2} \frac{|f(p^j)|}{p^j} < \infty. \] If $\kappa \leq 1$, suppose further that \[ \sum_{p^j \leq x} |f(p^j)| \ll_{f} x/\log{x} \quad \text{{\rm(}for $x\geq 2${\rm)}}.\]
Finally, suppose that \[\sum_{p}\frac{1}{p}(|f(p)|-\Re(f(p))) < \infty.\]Then as $x\to\infty$,
\begin{equation}\label{eq:wirsingasymptotic} \sum_{n \leq x} f(n) \sim \frac{\e^{-\gamma \kappa}}{\Gamma(\kappa)} \frac{x}{\log{x}} \prod_{p \leq x}\left(1+\frac{f(p)}{p} + \frac{f(p^2)}{p^2} + \dots\right). \end{equation}
Here $\gamma$ is the Euler--Mascheroni constant, and $\Gamma(\cdot)$ is the classical Gamma-function.
\end{prop}

\begin{proof}[Proof of the existence of the limiting distribution $F$] The characteristic function $\phi_x$ of $F_x$ is given by
\[ \phi_x(t) = \frac{1}{S(f;x)} \sum_{n \leq x} f(n) (n/\sigma(n))^{\i t}. \]
Because of the conditions on $f$ in Theorem \ref{thm:nonneg}, Proposition \ref{prop:wirsing} yields an asymptotic formula for $S(f;x)$. Proposition \ref{prop:wirsing} may also be applied to give an analogous formula for the partial sums of $f(n) (n/\sigma(n))^{\i t}$. To see this, notice that $|f(n) (n/\sigma(n))^{\i t}|=f(n)$, and that
\[ (p/\sigma(p))^{\i t} -1 = \left|\exp\left(\i t \log\frac{p}{p+1}\right	) -1\right| \leq |t \log \frac{p}{p+1}| =|t| \log \frac{p+1}{p} \leq |t|/p, \]
so that
\begin{equation}\label{eq:fdiff} f(p) (p/\sigma(p))^{\i t}  = f(p) + O(|t|/p).\end{equation}
The hypotheses of Proposition \ref{prop:wirsing}, with the same $\kappa$ as in \eqref{eq:wcond}, are now easily seen to follow from the conditions assumed on $f$. Comparing the asymptotic estimates obtained from \eqref{eq:wirsingasymptotic} for $f(n)$ and $f(n) (n/\sigma(n))^{it}$, we find that as $x\to\infty$ with $t$ fixed,
\begin{equation}\label{eq:asymptoticrelation} \phi_x(t) \sim \prod_{p \leq x} \left(\left(\sum_{j=0}^{\infty} \frac{f(p^j)}{p^j}  (p^j/\sigma(p^j))^{\i t}\right) \cdot \left(\sum_{j=0}^{\infty} \frac{f(p^j)}{p^j}\right)^{-1}\right).\end{equation} For notational convenience, let us write
\[ \alpha_p(t) = \sum_{j=0}^{\infty} \frac{f(p^j)}{p^j} (p^j/\sigma(p^j))^{\i t}, \quad\text{and}\quad \Delta_p = \sum_{j=0}^{\infty} \frac{f(p^j)}{p^j}. \]
Note that $\Delta_p$ is finite for every $p$, by \eqref{eq:twoormore}. Since the terms in the series defining $\alpha_p(t)$ are bounded in absolute value by the corresponding terms in $\Delta_p$, the series for $\alpha_p(t)$ converges uniformly, and so $\alpha_p(t)$ is continuous everywhere. Let
\[ \eta_p = \sum_{j=2}^{\infty} \frac{f(p^j)}{p^j}. \]
We will show below that for all primes $p$ exceeding a certain constant $p_0$,
\begin{equation}\label{eq:alphabeta} \alpha_p(t) \Delta_p^{-1} = 1 + O\left(\frac{1+|t|}{p^2} + \eta_p\right); \end{equation}
we allow both $p_0$ and the implied constant to depend on $f$. Now $\sum_{p} \frac{1}{p^2} < \infty$, and \eqref{eq:twoormore} asserts that $\sum_{p} \eta_p < \infty$. Assuming for the time being that \eqref{eq:alphabeta} has been established, we see that the series $\sum_{p > p_0} |\alpha_p(t) \Delta_p^{-1}-1|$ converges uniformly on any interval $[-T,T]$. Consequently, the infinite product
\[ \prod_{p > p_0} \alpha_p(t) \Delta_p^{-1} \]
converges to a function of $t$ that is continuous everywhere. Of course, the finite product $\prod_{p \leq p_0} \alpha_p(t) \Delta_p^{-1}$ is also continuous on all of $\R$. We conclude from \eqref{eq:asymptoticrelation} that as $x\to\infty$,
\[ \phi_x(t) \to \psi(t), \]
where
\begin{equation}\label{eq:psidef} \psi(t) := \prod_{p} \left(\left(\sum_{j=0}^{\infty} \frac{f(p^j)}{p^j}  (p^j/\sigma(p^j))^{\i t}\right) \cdot \left(\sum_{j=0}^{\infty} \frac{f(p^j)}{p^j}\right)^{-1}\right)  \end{equation}
is continuous everywhere. So by L\'evy's criterion, the $F_x$ converge weakly to a limiting distribution $F$ with characteristic function $\psi$.

It remains to establish the estimate \eqref{eq:alphabeta}. Using \eqref{eq:fdiff} once more, we find that
\begin{align}\notag\alpha_p(t) \Delta_p^{-1} &= \left(1+\frac{f(p)}{p} (p/\sigma(p))^{\i t} + O(\eta_p)\right) \Delta_p^{-1}\\
&=  \left(1+\frac{f(p)}{p} + O(|t|/p^2)\right)\Delta_p^{-1} + O(\eta_p).\label{eq:alphabeta0} \end{align}
Now $\Delta_p = 1 + \frac{f(p)}{p} + \eta_p$. We are assuming that $f(p)=O(1)$ and that $\sum_{p} \eta_p$ converges; thus, we can choose $p_0$ so that $0 \leq \Delta_p -1 \leq \frac12$ for all $p > p_0$. Since $\frac{1}{1+z}=1-z+O(z^2)$ for $|z| \leq \frac12$, we have for $p > p_0$ that
\begin{align}\notag \Delta_p^{-1} &= 1 - (\Delta_p-1) + O((\Delta_p-1)^2) \\
&= 1-\frac{f(p)}{p} + O\left(\eta_p + \frac{1}{p^2}\right).\label{eq:alphabeta1}
\end{align}
Substituting \eqref{eq:alphabeta1} into \eqref{eq:alphabeta0} yields \eqref{eq:alphabeta}.
\end{proof}

\subsection{Continuity} Let $X_p$ denote the discrete random variable taking the value $\log \frac{f(p^j)}{p^j}$ with probability $\frac{1}{\Delta_p} \cdot \frac{f(p^j)}{p^j}$, for each $j=0, 1, 2, \dots$.
Let $\phi_{X_p}$ be the characteristic function of $X_p$. Then
\begin{align*} \phi_{X_p}(t) = \mathbf{E}[\e^{\i t X_p}] &= \sum_{j=0}^\infty \e^{\i t \log\frac{p^j}{\sigma(p^j)}} \cdot \Prob\left(X_p = \log\frac{p^j}{\sigma(p^j)}\right) \\&=
\left(\sum_{j=0}^{\infty} \frac{f(p^j)}{p^j} (p^j/\sigma(p^j))^{\i t}\right) \left(\sum_{j=0}^{\infty} \frac{f(p^j)}{p^j}\right)^{-1}, \end{align*} which is precisely the $p${th} term in the product formula \eqref{eq:psidef}. This shows (cf. \cite[eq. (12)]{erdoswintner39}) that $\psi(t)$ is the infinite convolution of the $\phi_{X_p}$, as $p$ ranges over the primes. The following result of L\'{e}vy \cite[Lemma 1.22, p. 46]{elliott79} provides the approach that we will adopt in our proof that $\psi(t)$ is continuous.

\begin{lem}\label{lem:levy} Suppose that $\psi$ is an infinite convergent convolution of purely discontinuous distribution functions $\phi_1$, $\phi_2$, \dots; that is, $\psi = \phi_1 \ast \phi_2 \ast \cdots$. Let $d_k$ be the maximal jump of each $\phi_k$. If $\sum_{k=1}^{\infty} (1-d_k)$ diverges, then the limit distribution is continuous.
\end{lem} 

\begin{proof}[Proof of continuity] Let $d_p$ be the maximal jump in the distribution function of $X_p$. By Lemma \ref{lem:levy}, it suffices to show that $\sum_{p} (1 - d_p)$ diverges.  Now the distribution function of $X_p$ has jumps of size $\frac{1}{\Delta_p} \frac{f(p^j)}{p^j}$ at the points $\log \frac{p^j}{\sigma(p^j)}$, where $j$ ranges over those nonnegative integers with $f(p^j) \neq 0$. Taking $j=0$, we see that there is a jump at $x=0$ of size $\frac{1}{\Delta_p}$. Since $f(p)/p$ and $\eta_p$ both tend to zero, we may choose $p_0$ so that $\Delta_p = 1+ f(p)/p + \eta_p < 2$ for all $p > p_0$. For these values of $p$, we have $\frac{1}{\Delta_p} > \frac{1}{2}$, and so the largest jump must occur at $x=0$. Hence, $d_p = \frac{1}{\Delta_p}$ for $p > p_0$, and \begin{align*} \sum_p (1 - d_p) \geq \sum_{p > p_0} \frac{\Delta_p - 1}{\Delta_p}\geq \frac{1}{2} \sum_{p>p_0} (\Delta_p-1) \geq \frac{1}{2}\sum_{p > p_0}\frac{f(p)}{p}.\end{align*} Recall that $\sum_{p \leq x} \frac{f(p)}{p} \log p \sim \kappa \log x$ for a certain $\kappa > 0$. By partial summation,
\begin{equation}\label{eq:divergentsum} \sum_{p \leq x} \frac{f(p)}{p} \sim \kappa \log\log{x} \qquad\text{(as $x\to\infty$)}.\end{equation} Consequently, $\sum_p (1 - d_p)$ diverges.\end{proof}

\subsection{Strict monotonicity} Since we have already established the existence and continuity of $F$, we know at this point that $\tD_f$ is a well-defined, continuous function on $[0,1]$. Rather than prove that $F$ is strictly increasing on $(-\infty,0]$, we prove directly that $\tD_f$ is strictly increasing on $[0,1]$.

\begin{proof}[Proof that $\tD_f$ is strictly increasing]
It suffices to show that for $u, v \in [0,1]$ with $v< u$,
\begin{equation}\label{eq:liminfbit} \liminf_{x\to\infty} \frac{1}{S(f;x)} \sum_{\substack{n \leq x \\v < n/\sigma(n) \leq u}} f(n) > 0.\end{equation}
In proving \eqref{eq:liminfbit}, there is no loss of generality in assuming that $f$ is supported on squarefree integers. This is because $S(f;x)$ and $S(f\mu^2;x)$ have the same order of magnitude. To see this last claim, note that comparing the corresponding versions of \eqref{eq:wirsingasymptotic} shows that
\[ S(f; x) \sim S(f \mu^2;x) \cdot \prod_{p\leq x} \left(1 + \eta_p\left(1+\frac{f(p)}{p}\right)^{-1}\right), \]
as $x\to\infty$. Since $\sum_{p} \eta_p < \infty$, the right-hand product converges as $x\to\infty$. Thus, $S(f;x) \asymp S(f\mu^2;x)$ for large $x$, as claimed.

Since $f(p)$ is bounded, \eqref{eq:divergentsum} implies that the sum of the reciprocals of those $p$ with $f(p)\neq 0$ diverges. Since $|\log\frac{p}{\sigma(p)}|\asymp \frac{1}{p}$, we may use the greedy algorithm to select a squarefree natural number $m$ with $f(m) > 0$ and with $v < \frac{m}{\sigma(m)} \leq u$. We keep this $m$ fixed for the remainder of the argument. We let $y$ be a real parameter, viewed as fixed but eventually to be chosen very large. For now, we assume that $y$ exceeds the largest prime factor of $m$.

Consider the contribution to the sum in \eqref{eq:liminfbit} from those $n=mq$, where $q$ is squarefree and coprime to $\Pi_y:= \prod_{p \leq y} p$. We will show that if $y$ is chosen sufficiently large, then this contribution is already enough to imply \eqref{eq:liminfbit}. Notice that
\begin{align}\notag \sum_{\substack{q \leq x/m \\ \gcd(q,\Pi_y)=1\\v < mq/\sigma(mq) \leq u}} f(mq) &= f(m) \sum_{\substack{q \leq x/m \\ \gcd(q,\Pi_y)=1\\q/\sigma(q) > v\sigma(m)/m}} f(q) \\
		&\geq f(m) \sum_{\substack{q \leq x/m \\ \gcd(q,\Pi_y)=1}} f(q)\left(1-v\frac{\sigma(m)}{m}\cdot \frac{\sigma(q)}{q}\right). \label{eq:last}\end{align}
Let $\1_y$ be the indicator function of those numbers coprime to $\Pi_y$. Set $a_y(n) = f(n)\1_y(n)$ and $b_y(n) = f(n) \frac{\sigma(n)}{n}\1_y(n)$. The sum in \eqref{eq:last} can be written as
\begin{equation}\label{eq:sadiff} S(a_y; x/m) - v\frac{\sigma(m)}{m} S(b_y; x/m).\end{equation}
By Proposition \ref{prop:wirsing} and our assumption that $f$ is supported on squarefrees,
\begin{equation}\label{eq:sfxbound} S(f;x) \sim  \frac{\e^{-\gamma \kappa}}{\Gamma(\kappa)} \frac{x}{\log{x}} \prod_{p \leq x}\left(1+\frac{f(p)}{p}\right). \end{equation}
Since $f(p)$ is bounded, the asymptotic relation \eqref{eq:sfxbound} remains valid even if the product is shortened to be over the primes $p \leq x/m$. Now applying Proposition \ref{prop:wirsing} to $a_y$, we find that
\begin{align}\notag S(a_y; x/m) &\sim \frac{\e^{-\gamma \kappa}}{\Gamma(\kappa)} \frac{x}{m\log{x}} \prod_{y< p \leq x/m}\left(1+\frac{f(p)}{p}\right) \\
&\sim \frac{1}{m} S(f;x) \cdot \prod_{p \leq y}\left(1+\frac{f(p)}{p}\right)^{-1}.\label{eq:saasymp}
\end{align}
Similarly,
\begin{equation}\label{eq:sbasymp} S(b_y; x/m) \sim \frac{1}{m} S(f;x) \cdot \prod_{p \leq y}\left(1+\frac{f(p)}{p}\right)^{-1} \prod_{y < p \leq x/m} \frac{1+\frac{f(p)}{p}(1+\frac{1}{p})}{1+\frac{f(p)}{p}}. \end{equation}
Combining \eqref{eq:last}, \eqref{eq:sadiff}, \eqref{eq:saasymp}, and \eqref{eq:sbasymp}, we see that the $\liminf$ in \eqref{eq:liminfbit} is bounded below by
\begin{multline}\label{eq:crazybound} f(m) \cdot \liminf_{x\to\infty} \left(\frac{S(a_y; x/m)}{S(f;x)} - v\frac{\sigma(m)}{m} \frac{S(b_y; x/m)}{S(f;x)}\right) \\ =
\frac{f(m)}{m} \left(\prod_{p \leq y}\left(1+\frac{f(p)}{p}\right)^{-1}\right) \left(1-v\frac{\sigma(m)}{m}\prod_{p>y}\frac{1+\frac{f(p)}{p} (1+\frac{1}{p})}{1+\frac{f(p)}{p}}\right). \end{multline}
In the product over $p>y$, each term is at least $1$ but at most $1+f(p)/p^2 \leq 1 + O(1/p^2)$. Thus, that product tends to $1$ as $y\to\infty$. It follows that if we fix $y$ to be sufficiently large, then \eqref{eq:crazybound} is positive. This completes the proof. \end{proof}

\providecommand{\bysame}{\leavevmode\hbox to3em{\hrulefill}\thinspace}
\providecommand{\MR}{\relax\ifhmode\unskip\space\fi MR }
\providecommand{\MRhref}[2]{%
  \href{http://www.ams.org/mathscinet-getitem?mr=#1}{#2}
}
\providecommand{\href}[2]{#2}

\end{document}